\documentclass[11pt
]{amsart}
\usepackage{latexsym}
\usepackage{amssymb, amsmath}
\usepackage{amsthm}
\usepackage{geometry}
\usepackage{hyperref}
\usepackage{changebar}
\usepackage{comment}
\usepackage[T1]{fontenc}
\usepackage[stretch=10]{microtype}
\usepackage{tikz}
\usetikzlibrary{calc}
\usetikzlibrary{intersections}
\usepackage{fixltx2e}
\usepackage{wrapfig}

\newcommand{\AAA}{\mathbb{A}}

\newcommand{\CC}{\mathcal{C}}

\newcommand{\LL}{\mathcal{L}}

\newcommand{\MM}{\mathcal{M}}
\newcommand{\NN}{\mathcal{N}}

\def\SS{\mathcal{S}}

\newcommand{\id}{{\rm id}}

\newcommand{\real}{\mathbb{R}}

\newtheorem{proposition}{Proposition}[section]

\newtheorem{theorem}[proposition]{Theorem}

\theoremstyle{remark}
\newtheorem{remark}[proposition]{Remark}
\theoremstyle{definition}

\numberwithin{equation}{section}

\begin{document}

\title{Linear response for intermittent maps}
\author{Viviane Baladi
\and Mike Todd
}
\address{D.M.A., UMR 8553, \'Ecole Normale Sup\'erieure,  75005 Paris, France}
\curraddr{Sorbonne Universit\'es, UPMC Univ Paris 06, CNRS, Institut de Math\'ematiques de Jussieu-Paris Rive Gauche (IMJ-PRG), Analyse Alg\'ebrique, 4, Place Jussieu, 75005 Paris, France}

\email{viviane.baladi@imj-prg.fr}

\address{Mathematical Institute, University of St Andrews, 
North Haugh,
St Andrews, 
KY16 9SS, Scotland}
\email{m.todd@st-andrews.ac.uk}

\date{}
\begin{abstract}We consider the one parameter family $\alpha \mapsto T_\alpha$
($\alpha \in [0,1)$)
of Pomeau-Manneville type interval maps
$T_\alpha(x)=x(1+2^\alpha x^\alpha)$ for  $x \in [0,1/2)$
and $T_\alpha(x)=2x-1$ for $x \in [1/2, 1]$,  with the associated
absolutely continuous invariant probability measure
$\mu_\alpha$. For $\alpha \in (0,1)$, Sarig and Gou\"ezel proved that  the system
mixes only polynomially with rate $n^{1-1/\alpha}$ (in particular, there
is no spectral gap). We show that for any $\psi\in L^q$, 
the map $\alpha \to \int_0^1 \psi\, d\mu_\alpha$
is differentiable on $[0,1-1/q)$, and we give a (linear response) formula
for the value of the derivative.  This is the first time that a linear response formula
for the SRB measure
is obtained in the setting of slowly mixing dynamics. Our argument shows how cone techniques can
be used in this context. For $\alpha \ge 1/2$ we need the $n^{-1/\alpha}$ decorrelation obtained
by Gou\"ezel under additional conditions. 
\end{abstract}
\thanks{This work was started in 2014 during a visit of MT to DMA-ENS, continued during
a visit of VB to St Andrews in 2015, and finished during a stay
of VB in the Centre for Mathematical Sciences in Lund. We are grateful to these institutions for their
hospitality, and we thank I. Melbourne for pointing out reference \cite{Tha00} and
A. Korepanov for inciting us to sharpen our results.
VB thanks  H.~Bruin for explanations on \cite{BT} and T. Persson for a conversation
on $L^q$. 
She is much indebted to S. Gou\"ezel for  pointing out Theorem 2.4.14 in \cite{Goth}, which
allowed us to extend our results to $\alpha \in[1/2,1)$.}
\maketitle

\section{Introduction}

Given a family of dynamical systems $T_\alpha$ on
a Riemann manifold, depending smoothly
on a real parameter $\alpha$, and admitting (at least for some large subset
of parameters) an ergodic physical (e.g. absolutely continuous, or SRB) 
invariant measure $\mu_\alpha$, it is natural to ask how smooth
is the dependence of $\mu_\alpha$ on the parameter
$\alpha$. In particular, one would like to know whether $\alpha \mapsto
\mu_\alpha$ is differentiable and, if possible, compute a formula
for the derivative, depending on $\mu_\alpha$, $T_\alpha$,
and $v_\alpha=\partial_\alpha T_\alpha$. 

This theme of {\it  linear response} was explored
in a few pioneering papers \cite{Ru1,KKPW,Ru} in the setting of smooth hyperbolic
dynamics (Anosov or Axiom~A), and then further developed, following the
influence of ideas of David Ruelle. In the smooth hyperbolic case, the SRB
measure $\mu_\alpha$ corresponds to the fixed point of a transfer
operator $\LL_\alpha$ enjoying a spectral gap on
a suitable Banach space. In particular, this fixed point is
a simple isolated eigenvalue in the spectrum of $\LL_\alpha$, and
linear response can be viewed as an instance of perturbation theory
for simple eigenvalues. This is evident in the linear response
formulas, which all involve some avatar of the
resolvent $(\id-\LL_\alpha)^{-1}=\sum_k \LL_\alpha^k$ applied to a suitable vector $Y_\alpha$, depending
on the derivative of $\mu_\alpha$ and on $v_\alpha$.

It was soon realised that
existence of a spectral gap is not sufficient to
guarantee linear response when bifurcations
are present (see e.g. \cite{Ma, B1, BS}). In the other direction, neither the
spectral gap nor structural stability 
is necessary for linear response, as was shown by Dolgopyat \cite{Do}
who obtained a linear response formula for some rapidly mixing
systems (which were not all exponentially mixing or
structurally stable).

The intuition that
a key sufficient condition is convergence of the
sum $\sum_k \LL_\alpha^k(Y_\alpha)$ 
was confirmed by  \cite[Remark 2.4]{HM}. This is
of course related to a summable decay of correlations. However,
decay of correlation usually only holds for observables with a suitable
modulus of continuity, which $Y_\alpha$, being a derivative, does not always
enjoy.
We confirm this intuition by  studying a toy-model, of Pomeau-Manneville type:
\footnote{See Remark~ \ref{amazing}  for one possible generalisation.}
For $\alpha\in [0,1)$, we consider the maps (as in \cite{LSV}) 
$T_\alpha: [0,1]\to [0,1]$:
$$
T_\alpha(x)=\begin{cases}
x(1+2^\alpha x^\alpha)\,  ,& x \in [0,1/2)\\
2x-1\, , & x \in [1/2, 1]\, .
\end{cases}
$$
(Of course, $T_0$ is just the angle-doubling map $T_0(x)=2 x$ modulo $1$.)
It is well-known that  each such $T_\alpha$ admits a unique absolutely continuous invariant
probability measure $\mu_\alpha=\rho_\alpha\, dx$.
(Clearly, $\rho_0(x)\equiv 1$.)
Statistical stability (continuity) of $\mu_\alpha$ when $\alpha$ changes is proved
in \cite{FT}. 
The absolutely continuous invariant
probability measure $\mu_\alpha=\rho_\alpha dx$ is mixing for all
$\alpha \in [0,1)$. For $\alpha=0$ the mixing rate for Lipschitz observables,
say, is exponential (decaying like $1/2^k$). For $\alpha \in (0,1)$ the mixing rate
is only polynomial with rate
$n^{1-1/\alpha}$ \cite{Go, Sa}. (In fact, Gou\"ezel obtains a faster rate
$n^{-1/\alpha}$ for $\int (\psi\circ T_\alpha^n) \phi\, d\mu_\alpha$,
if $\psi$ is bounded, $\phi$ is Lipschitz and vanishes in a neighbourhood
of zero, and $\int \phi\, d\mu_\alpha=0$, and  this property is crucial below when $\alpha \ge 1/2$.) In particular, for any $\alpha\in (0,1)$,
the density
$\rho_\alpha$ cannot be the fixed point of a transfer operator with a spectral gap on
a Banach space containing all $\CC^\infty$ functions.
However, we are able to prove (Theorem ~\ref{main}) that for any $q\in[1,\infty]$ and any 
$\psi\in L^q$, the map
$
\alpha \mapsto \int \psi \, d\mu_\alpha
$
is continuously differentiable on $[0,1-1/q)$, and we give two expressions (\eqref{LRF}, with
a resolvent, \eqref{susc}, of susceptibility function type) for the linear
response formula, with $Y_\alpha=(X_\alpha \NN_\alpha(\rho_\alpha))'$, where
$X_\alpha = v_\alpha \circ (T_\alpha|_{[0,1/2)}^{-1})$ and $\NN_\alpha$ corresponds to the first
branch of the transfer operator  $\LL_\alpha$.
This is the first time that a linear response formula
is achieved for a slowly mixing dynamics. The fact that linear response holds for any bounded $\psi$
is relevant since nonsmooth observables appear naturally.  For example, 
if $A$ is smooth and $\Theta$ is a Heaviside function,
the expectation value
of $\Theta(A(x))$ gives the fraction of the total measure where $A$ has
positive value, and more generally such discontinuous observables
 have probabilistic and physical interpretations, with the work of Lucarini et al. \cite{Luca, Luca2} 
showing how the theory of extremes for dynamical systems (in particular
 regarding climate change)  can
be cast in this framework.

Our proof is based on the cone techniques
from \cite{LSV}, and hinges on the new observation that the factor $X_\alpha$, respectively
$X_\alpha'$, compensates the singularity at zero of $\rho_\alpha'$, respectively $\rho_\alpha$. 
Indeed, the compensation is drastic enough so that the $n^{-1/\alpha}$ decorrelation results
of Gou\"ezel \cite{Go, Goth} can be used.
It would apply e.g. to the more general one-dimensional maps with finitely many neutral
points described in \cite[Section 5]{LSV}.
Since our goal is to describe a new mechanism (demonstrating in particular
how  invariant cone techniques
\footnote{Bomfim et al.  \cite{BomCasVar12} use invariant cones to obtain
differentiability of some equilibrium measures (enjoying a spectral gap)
of Pomeau--Manneville
maps. Their results do not apply to the SRB measure, and thus do not 
include linear response in the sense of the present
work.} can be implemented) for linear response in the presence of
neutral fixed points in the simplest setting,  we leave such generalisations to further works.

\medskip
 We end this introduction with comments  about bifurcations and the
 singularities of $\rho_\alpha$.

By   \cite[Theorem 1]{Th}, for any $\alpha \in (0,1)$, there exists $0<c_1<c_2$ so that 
\footnote{The upper bound also follows from \cite[Lemma 2.3]{LSV}.}
\begin{equation}
\label{basic}
c_1 x^{-\alpha} \le \rho_\alpha(x)\le c_2 x^{-\alpha}
\, .
\end{equation}
It is easy to see (e.g. via symbolic dynamics)
that the maps $\{T_\alpha\mid \alpha \in [0,1)\}$ belong to the same topological class,
so that bifurcations do not occur. 
However the conjugacy $h_\alpha$ between $T_0$ and $T_\alpha$
is
not differentiable. Indeed, if it were, then we would have $T_\alpha'(0)=T'_0(0)$
at the fixed point $0$, but this is impossible
\footnote{Also, there are many periodic points $x_0$  for $T_0$ such that if $p$ is the period then
 $(T_0^p)'(x_0)= 2^p$, but $(T_\alpha^p)'(h_\alpha(x_0))\ne  2^p$.}  since $T'_0(0)=2$,
while $T_\alpha'(0)=1$ for $\alpha >0$. More generally, take arbitrary $\alpha \ne \beta$.
 The conjugacy $h_{\beta,\alpha}$ between $T_\alpha$ and $T_\beta$ maps the invariant density $\rho_\alpha$ to $\rho_\beta$. Therefore $h_{\beta,\alpha}$  cannot be differentiable, since otherwise
it would contradict \eqref{basic}.

Another lesson of recent research \cite{Ru2, Ru3, B3, B2, CD}  on linear response is
that understanding the singularities of the SRB measure is essential. 
In our application, the density 
$\rho_\alpha$ is smooth 
on $(0,1]$.
The only ``critical point'' of $T_\alpha$ is the neutral fixed point at $0$, so that the ``postcritical
orbit'' is reduced to a single point.  
By \eqref{basic}, the singularity type of $\rho_\alpha$ at $0$ is $x^{-\alpha}$. So, heuristically,
for a bounded observable $\psi$, the contribution of the origin to 
$\partial_\alpha \int \psi d \mu_\alpha$ should be
$\int x^{-\alpha} \log x \cdot \psi(x) \, dx$, which is indeed well defined.\footnote{While this paper was being finished, it was pointed out to us by I. Melbourne that this heuristic argument can be made rigorous for the special family of maps studied by Thaler in \cite[Section 2]{Tha00}, where
the invariant density takes an explicit form.}
Indeed, this heuristic remark sheds some light on the otherwise
mysterious singularity  cancellation $X_\alpha \rho'_\alpha \sim
 \log x \sim X'_\alpha \rho_\alpha$.
Our approach should extend to give higher order derivatives of $\alpha \to \mu_\alpha$
(using invariant cones with more derivatives).

\medskip
After the first version of this paper 
(in which our result was restricted to $\alpha \in [0,1/2)$ and $L^\infty$ observables)
was posted on the
arXiv, Korepanov     \cite{Ko} obtained linear response
(without the formula) for all $\alpha \in
(0,1)$ and $L^q$ observables
 (for $q> (1-\alpha)^{-1}$). His method of proof (using inducing) is
 different from ours.

\section{Linear response formula for Pomeau--Manneville maps}

\subsection{Statement of the main result}

We  consider the transfer operator $\LL_\alpha$ defined,
e.g. on $L^\infty(dx)$, by
(note that $\inf T'_\alpha \ge 1$ so absolute values are not needed)
$$
\LL_\alpha \varphi(x)=\sum_{T_\alpha(y)=x}
\frac{\varphi(y)}{T'_\alpha(y)} \, .
$$
(For functions depending both on $\alpha$ and
$x$ we denote by  $'$ the derivative w.r.t.\  $x$ and $\partial_\alpha$ the
derivative w.r.t.\ $\alpha$.)

We introduce some  notation in order to state
our main result. Let $f_\alpha :[0,1/2]\to [0,1]$ and $g_\alpha: [0,1]\to [0,1/2]$ be defined by
\begin{equation}\label{invv}
f_\alpha(x)=x(1+2^\alpha x^\alpha)\, , \quad g_\alpha(y)= f^{-1}_\alpha(y)  \, .
\end{equation}
Note that  $g_0(y)=y/2$, while for  $\alpha >0$ we have
\begin{equation}\label{useful}
|g_\alpha(y)- y(1-2^\alpha y^\alpha)|\le C(y^{1+2\alpha})\, , \forall y\in [0,1]\, .
\end{equation} 

For $0\le x\le 1/2$ and $0\le \beta \le 1$, we have
$
v_\beta(x):=\partial_\beta T_\beta(x)=2^\beta x^{1+\beta}\log(2 x)$.
Therefore, for $0\le x\le 1$ and $0\le \beta \le 1$, 
\begin{equation}\label{Xb}
X_\beta:=v_\beta\circ  g_\beta=2^\beta g_\beta^{1+\beta}\log (2g_\beta)\, ,
\qquad\qquad
|X_\beta(x)|\le c x^{1+\beta} (|\log x|+1)\, ,
\end{equation}
and
\begin{equation}\label{X'b}
X'_\beta=2^\beta g_\beta' g_\beta^\beta\left[(1+\beta)  \log (2g_\beta)+1\right]\, , 
\qquad\quad \, \,\,   |X'_\beta(x)| \le  c x^{\beta} (|\log x|+1)\, ,
\end{equation}
and, finally, 
\begin{align}
\label{X''b}
  X''_\beta=2^\beta g_\beta^{\beta-1} \bigl [ g_\beta'' g_\beta  
  \big((1+\beta)\log (2g_\beta)+1\big) 
&+  (g_\beta')^2  (  (\beta+\beta^2)\log (2g_\beta)+ 1+2\beta)\bigr ],\\
\nonumber & \qquad
  |X''_\beta(x)|\le  c x^{\beta-1} (|\log x|+1)\, . 
  \end{align}
(The properties of $X_\beta$
and its derivatives above are  at the heart of the mechanism of the proofs.)
Since $T_\alpha(x)$ is independent from $\alpha$
if $x>1/2$, to state our main result,
we  need the  transfer operator associated to the first branch
of $T_\alpha$ by
$$
\NN_\alpha \varphi(x)=g_\alpha'(x)\cdot  \varphi(g_\alpha(x)) \, .
$$

\begin{theorem}[Linear response formula]\label{main}
Let $\alpha \in (0, 1)$. Then for any $q> (1-\alpha)^{-1}$
and any  $\psi \in L^q(dx)$
\begin{equation}\label{LRF}
\lim_{\epsilon \to 0}
\epsilon ^{-1}\biggl ( \int_0^1 \psi \, d\mu_\alpha - \int_0^1 \psi \, d\mu_{\alpha +\epsilon}
\biggr ) 
=- \int_0^1 \psi (\id -\LL_\alpha)^{-1} \bigl [ (X_\alpha\NN_\alpha (\rho_\alpha))'\bigr ] \, 
dx\, .  \end{equation}
(In particular the  right-hand side of \eqref{LRF} is well-defined.)
In addition, the  right-hand side of \eqref{LRF}
can be written as the following absolutely convergent sum 
\begin{equation*}
-\sum_{k \ge 0} \int_0^1 \psi \LL_\alpha^k \bigl [ (X_\alpha\NN_\alpha (\rho_\alpha))'\bigr ]
dx=
-\sum_{k \ge 0} \int_0^1 (\psi \circ T_\alpha^k)  (X_\alpha\NN_\alpha (\rho_\alpha))'
dx
\, .
\end{equation*}
The result also holds for $\alpha=0$, taking the limit as  $\epsilon\downarrow 0$ in \eqref{LRF}.
For $p\in[1,\infty)$, the map $\alpha \mapsto \partial_\alpha \rho_\alpha\in L^{p}(dx)$ is 
continuous on $[0,1/p)$.
\end{theorem}

Integration by parts allows us to rewrite the convergent sum as
\begin{equation}\label{susc}
-\sum_{k \ge 0} \int_0^1 (\psi \circ T_\alpha^k)  (X_\alpha\NN_\alpha (\rho_\alpha))'
dx\\
=
\sum_{k \ge 0} \int_0^1 (\psi \circ T_\alpha^k)'  X_\alpha\NN_\alpha (\rho_\alpha) dx
\, .
\end{equation}
We conjecture that the above results also hold for $\alpha <0$ in some parameter
range, but the proof will require modifications.

\begin{remark}\label{amazing}
It seems our proof also applies to the one-parameter family
$$
G_t(x)=\begin{cases}
(1+t)  x+ (1-t) 2^t x^{1+t}\, ,  & x\in [0,1/2)\, ,\\
2x-1 \, , & x\in [1/2,1] \, ,
\end{cases} \qquad\qquad
t\in [0, \epsilon] \, . 
$$
\end{remark}

\begin{remark}
For the sake of comparison with previous works
 (e.g. \cite{B3}), we can consider a one-parameter family $F_\beta$ obtained
by perturbations in the image, i.e.
so that $v_\beta =\partial_\beta F_\beta =X_\beta \circ F_\beta$ for some
$X_\beta$. This
can be achieved by perturbing the first branch ($x<1/2$)  in order
to have $T_\beta=F_\beta$ and requiring the second branch ($x\ge 1/2$)
to move ``sympathetically'' with the first one. More precisely, for fixed $\alpha$,
consider the one-parameter family $F_{\beta,\alpha}$  which satisfies
$$
F_{\alpha,\alpha}(x)=T_\alpha(x)\, , \forall x\in [0,1]\, ,
\qquad
F_{\beta, \alpha}(x)=T_\beta(x) \, , \forall x\in [0,1/2) \, ,
$$
 and, setting 
$$ 
X_{\beta}(x)=:v_\beta \circ g_\beta(x)=\partial_\beta T_\beta \circ g_\beta(x)
\, , \forall x\in [0,1]\, ,$$
so that
$$
v_{\beta, \alpha}:= \partial_\beta F_{\beta, \alpha}=X_{\beta} \circ F_{\beta, \alpha} \, .
$$
(For $x \in [0,1/2)$ this is automatic, and for $x\in [1/2,1]$ it can
be obtained by solving the ODE
$\partial_\beta F_{\beta,\alpha}= X_\beta \circ F_{\beta,\alpha}$
with initial condition $F_{\alpha,\alpha}(x)=2x-1$ on $[1/2,1]$. By the Picard-Lindel\"of theorem, this ODE has a unique solution
since $(\beta,y)\mapsto X_\beta (y)$ is continuous in $\beta$.)
If $\alpha$ is fixed, slightly abusing notation, we sometimes write $F_\beta$, $v_\beta$, $X_\beta$,
and  $\LL_\beta$,
instead of $F_{\beta,\alpha}$, $v_{\beta,\alpha}$, $X_{\beta,\alpha}$,
and $\LL_{\beta,\alpha}$,
when the meaning is clear. It is not difficult to prove that $F_\beta$ has a
unique absolutely continuous invariant measure  $\hat\mu_\beta =\hat \rho_\beta\, dx$
satisfying the same properties as $\mu_\beta$,
and the proof of Theorem~\ref{main} shows that
for any $\alpha \in (0, 1)$ and any $\psi \in L^\infty(dx)$
\begin{equation}\label{LRF'}
\lim_{\epsilon \to 0}
\epsilon ^{-1}\biggl ( \int_0^1 \psi \, d\hat \mu_\alpha - \int_0^1 \psi \, d\hat \mu_{\alpha +\epsilon}
\biggr ) 
=- \int_0^1 \psi (\id -\LL_\alpha)^{-1} \bigl [ (X_\alpha\hat  \rho_\alpha)'\bigr ] 
dx \, . \end{equation}
The result also holds for $\alpha=0$, taking the limit as  $\epsilon\downarrow 0$ in \eqref{LRF'}.

Just like \eqref{LRF}, the expression \eqref{LRF'} can be written as an absolutely convergent sum.
Integration by parts gives
$\int (\psi' \circ T_\alpha^k ) \cdot   (T_\alpha^k )'(x) X_\alpha(x)  \rho_\alpha(x) \, dx$.
This is just $\Psi(1)$ where $\Psi(z)$ is the susceptibility function (see e.g. \cite{B3}).
It would be interesting to analyse  the singularity 
type of the susceptibility function at $z=1$. (See \cite{BMS} for the corresponding analysis
for piecewise expanding maps.)
\end{remark}

\subsection{Invariant cones} 
\label{invc}

  Before, proving the theorem, we introduce notations
and state
useful results regarding cones adapted from \cite{LSV}.
 
As our proof requires higher derivatives we shall use the following fact:

\begin{proposition}[Invariant cone in $\CC^2$]\label{cheat'}
For fixed 
$b_1\ge \alpha+1$, $b_2\ge b_1$,    $\bar b_1>0$, $\bar b_2>0$, 
define the cone $\CC_2$ to be the set of $\varphi \in \CC^2(0,1]$ so that
\begin{align*} 
&\varphi(x) \ge 0\, , 
\, \,  \, \frac{\bar b_1}{x}  \varphi(x)\le -\varphi'(x) \le \frac{b_1}{x}  \varphi(x)\, , 
\, \text{ and } \, \, 
 \frac{\bar b_2}{x^2}  \varphi(x)\le \varphi''(x)\le \frac{b_2}{x^2}  \varphi(x)\, ,
 \,  \,  \forall x \in (0,1] \,  .
\end{align*}
Then there exists $b_{max}<\infty$
so that for any $0\le \alpha <1$   there exists $\alpha+1 \le b'(\alpha) <b_{max}$
and $\bar b >1/b_{max}$ so that
if  $b_1\ge \alpha+1$, $ b_2 \ge b'(\alpha)$,  $\max\{ \bar b_1, \bar b_2\}
\le 1/b'(\alpha)$ we have 
\begin {equation}
\label{miracle2}
\varphi \in \CC_2 \, 
 \Longrightarrow\, \, \,  \LL_\alpha(\varphi )\in \CC_2
\mbox{ and } 
\NN_\alpha (\varphi )\in \CC_2 \, .
\end{equation}
\end{proposition}

The proof of Proposition~\ref{cheat'} is given in Appendix~\ref{theproof}.

For $\varphi \in L^1(dx)$ we set
$
m(\varphi)=\int_0^1 \varphi(x)\, dx
$.
For  $a\ge 1$,  we denote by  $\CC_{*}=\CC_*(\alpha, a)$  the  cone
\begin{equation}
\label{Cstar}
\CC_{*}:=\left\{ \varphi \in \CC^1(0,1]\mid 
0\le \varphi(x) \le 2 a \rho_\alpha(x) m(\varphi) \, , \, \, 
  - \frac{\alpha+1}{x}  \varphi(x)\le \varphi'(x) \le 0\, , \, \, 
  \forall x \in (0,1] \right\}\, .
\end{equation}
By \cite[Lemma 2.2]{LSV}, we have
\footnote{Noting that  $- \frac{\alpha+1}{x}  \varphi(x)\le \varphi'(x) \le 0$
if and only if $\varphi$ is decreasing and $x^{\alpha+1}\varphi$ is increasing.}
\begin{equation}\label{Cstarii}
\LL_\alpha(C_{*}(\alpha,a))\subset \CC_{*}(\alpha,a)\, ,\quad
\forall a \ge 2^\alpha(\alpha+2)\, .
\end{equation}
Note also that by definition (this will be used to show \eqref{half'})
\begin{equation}\label{half}
\int_0^{1/2} \varphi\, dx \ge \frac{1}{2} m(\varphi) \, , \qquad \forall \varphi \in \CC_*\, .
\end{equation}
Finally,
for  $a\ge  2^\alpha(\alpha+2)$ and $b_1\ge \alpha+1$,  we denote by  $\CC_{*,1}=\CC_{*,1}(\alpha, a, b_1)$  the  cone
\begin{equation}
\label{Cstar1}
\CC_{*,1}:=\left\{ \varphi \in \CC^1(0,1]\mid 
0\le \varphi(x) \le 2 a \rho_\alpha(x) m(\varphi) \, , \, 
  |\varphi'(x)| \le \frac{b_1}{x}  \varphi(x)\, , \, 
  \forall x \in (0,1] \right\}\, .
\end{equation}
Note that    \eqref{basic} implies
\begin{equation}
\label{Cstar1_basic}
\varphi(x) \le \frac{ 2ac_2 }{ x^\alpha}m(\varphi) \, , \qquad
\forall \varphi \in \CC_{*,1}(\alpha)\, , \, \, \forall x\in (0,1]\, .
\end{equation}
By definition,  $\CC_{*,1}\subset L^1(dx)$, and
\begin{equation}\label{Cstari}
 \beta \ge \alpha\ge 0\, \Longrightarrow\, 
\, 
\CC_{*,1}(\alpha,a,b_1)\subset \CC_{*,1}\left(\beta,\frac{c_2 }{c_1}a,b_1\right) \, .
\end{equation}
Also, the arguments of \cite{LSV} give $a'(\alpha)$ so that
\begin{equation}\label{coned}
\rho_\alpha\in  \CC_*(\alpha) \cap\CC_{*,1}(\alpha)
\cap \CC_2
\end{equation}
 if the parameters satisfy $a\ge a'(\alpha)$,  $b_1\ge \alpha+1$, and $b_2 >b'(\alpha)$.

\begin{remark}[Cones $\CC_*$ and $\CC_{*,1}$.]
The definition of $\CC_{*,1}$ will make it
easy to check that $-(X_\alpha \varphi)'+c_\varphi$,  for suitable
 $\varphi \in \CC_{*,1}$ and constant $c_\varphi>0$, lies in a cone $\CC_{*,1}$ (possibly for larger
 $a$ and $b_1$), via  the Leibniz formula. In \S \ref{theprooft} this will
 be used to get bounds $\varphi(x)\le C x^{-\alpha}$ and
 $|\varphi'(x)|\le C x^{-1-\alpha}$, while in Appendix~\ref{coneonly} these cones
 play a more important role. This
 is why we  use
$\CC_{*,1}$ instead of the cone  $\CC_*$ 
used in \cite{LSV}. However, the condition that $\varphi(x)$ is decreasing will
be used to get \eqref{half'}.
(In \cite{LSV} the condition that $\varphi(x)$ be decreasing
 is only used in \cite[Lemma 2.1]{LSV}, which we do not need
in view of Proposition~\ref{cheat'} and \eqref{toremember1}.)
\end{remark}

We will use 
the following result
(see Appendix~\ref{theproof} for the proof):

\begin{proposition}[Invariance of the cone $\CC_{*,1}$]\label{replaces4.10}
Fix $\alpha \in (0,1)$, $a\ge 1$, and $b_1 \ge \alpha+1$.
Then
\begin{align}
\nonumber &\LL_\alpha(\CC_{*,1}(\alpha,a,b_1))\subset \CC_{*,1}(\alpha,a,b_1)\, , \\
\label{half'} &\NN_\alpha\left(\CC_{*,1}(\alpha,a,b_1)\cap \left\{\int_0^{1/2}\varphi \, dx \ge \frac{1}{2} m(\varphi)\right\}\right)\subset \CC_{*,1}(\alpha, 2a,b_1)\, .
\end{align}
In addition,  there exists  $C>0$, independent of $\alpha$, $a$,  and $b_1$ so that we have
for any  $\psi \in L^\infty$ and $\varphi\in \CC_{*,1}(\alpha)+ \real$ with
$\int \varphi \, dx=0$, 
\begin{equation}
\label{toremember}
\left|\int_0^1 \psi   \LL_0^k( \varphi )\, dx\right|
\le \frac{Ca b_1}{ (1-\alpha) (\log k)k^{-2+1/\alpha}}
\|\psi\|_{L^\infty}  \| \varphi \|_1\, , \, \,
\forall k \ge 1 \, .
\end{equation}
\end{proposition}

Note  that for fixed $a_{max}<\infty$ and $b_{max}<\infty$, the expression \eqref{toremember}
is controlled by
\begin{equation}\label{furtherr}
\sup_{ b_1\le b_{max}, \, a\le a_{max}}\,\, \, \frac{Ca b_1 }{ (1-\alpha)(\log k) k^{-2+1/\alpha}} <\infty\, .
\end{equation}

\subsection{Proof of Theorem ~ \ref{main}}
\label{theprooft}

We may now prove the theorem:

\begin{proof}[Proof of Theorem \ref{main}]

{\bf Step 0:} We show that the  right-hand side of \eqref{LRF} is well-defined
for bounded
$\psi$.
First observe that integration by parts and $X_\alpha(0)=X_\alpha(1)=0$ (because
$v_\alpha(0)=0$ and $v_\alpha(1/2)=0$)
imply
$$
\int_0^1 (X_\alpha \NN_\alpha(\rho_\alpha))' \, dx=0 \, .
$$
Then note that 
\begin{equation}
\label{L1ok}
\|(X_\alpha \NN_\alpha(\rho_\alpha))'\|_1=
\|X_\alpha (\NN_\alpha(\rho_\alpha))'
+ X'_\alpha \NN_\alpha(\rho_\alpha)\|_1
\le C \int_0^1(|\log x|+1)\, dx  <\infty \, .
\end{equation}
Indeed,  this is easy for $\alpha=0$ since
 \begin{equation}
 \label{X0}
 X_0(x)=\frac{x(\log 2+ \log(x/2))}{2}\, , \qquad
X'_0(x)=\frac{1+\log 2 + \log(x/2)}{2}\, ,
\quad \forall x\in [0,1]\, .
\end{equation}
Next, $\rho_0|_{[0,1]}\equiv 1$, with $(\NN_0 \rho_0)|_{[0,1]}\equiv 1/2$, so that
for any $x\in [0,1]$,
\begin{equation}
\label{X0N0}
(X_0 \NN_0(\rho_0))'(x)=\frac{X_0'(x)}{2}=\frac{ 1+\log 2 + \log(x/2)}{4}\, .
\end{equation}

For $\alpha >0$, on the one hand, \eqref{basic} 
and Proposition ~ \ref{cheat'} imply that 
$$
| \rho_\alpha' (x)|\le ac_2 x^{-1-\alpha}\, ,\, 
\NN_\alpha(\rho_\alpha)(x)\le \rho_\alpha(x) \le c_2 x^{-\alpha}\, ,
 \, \, |(\NN_\alpha(\rho_\alpha))'(x)|\le  c_2 b_1 x^{-1-\alpha}\, .
 $$
On the other hand, the dominant term of
$X_\alpha(x)$ is a constant multiple
of $x^{1+\alpha}\log x$ (see \eqref{Xb}) while the dominant term of $X'_\alpha(x)$ is
a constant multiple of
$x^{\alpha}\log x$ (see \eqref{X'b} and recall \eqref{useful}). This establishes
\eqref{L1ok} for $\alpha>0$.

Next, write the  right-hand side of \eqref{LRF}  as
\begin{align}
\label{parts}
\left|\sum_{j=0} ^\infty \int_0^1 \psi \cdot \LL^j_\alpha [ (X_\alpha \NN_\alpha( \rho_\alpha))'] \, dx\right|&
\le 
\sum_{j=0} ^\infty \left|\int_0^1 \psi \cdot \LL^j_\alpha [ (X_\alpha \NN_\alpha( \rho_\alpha))'] \, dx\right| \, .
\end{align}
The function $f=-(X_\alpha \NN_\alpha( \rho_\alpha))'/\rho_\alpha$ is H\"older, and it
vanishes at zero, with $\int f \rho_\alpha\, dx=0$. In addition,
for any $\epsilon \in (0,1)$ we have  $C$
so that $|f(x)|\le  C x^{\alpha(1-\epsilon/2)}$.  Since 
$$\frac{1}{\alpha}(1+\alpha(1-\epsilon/2))-1 > \frac{1}{\alpha}-\epsilon\, ,$$
if $\epsilon>0$ is small enough then
 \cite[Thm 2.4.14]{Goth} applied
 \footnote{This theorem is a strengthening of \cite[Prop. 6.11, Cor. 7.1]{Go}.} to
 $f=-(X_\alpha \NN_\alpha( \rho_\alpha))'/\rho_\alpha$ 
gives $K_\alpha>0$ 
so that
the $j$th term in the right-hand side of \eqref{parts}
is bounded by 
$$
 \sum_{j=0} ^\infty \left|\int_0^1 (\psi\circ T^j_\alpha) \cdot f \, d\mu_\alpha \right|
 \le CK_\alpha \|\psi\|_{L^\infty} 
\frac{1}{j^{(1/\alpha)-\epsilon}}  \, .
$$
Since we may take $\epsilon < (1/\alpha) -1$, this is summable.

If $\alpha=0$, fixing $\beta >0$,
it is easy to see that  there exists  a constant $C_X>0$  
 so that $-(X_0' \NN_0( \rho_0))'+C_X$  belongs
 to $\CC_{*,1}(\beta, a,b_1)$,  for suitable $a$ and $b_1$
(see \eqref{Cstar1}).
We may apply \eqref{toremember} 
from Proposition~\ref{replaces4.10}  to
$\varphi=-(X_0 \NN_0(\rho_0))'\in \CC_{*,1}+\real$
in order to bound the $j$th term in the right-hand side of
\eqref{parts}.

\smallskip

{\bf Step 1:}
Let $\psi$ be a bounded function so that  $\int \psi d\mu_\alpha=0$. 
We first show that $\beta \mapsto \int \psi \rho_\beta \, dx$ is Lipschitz at
$\beta=\alpha$.
Applying
the bound on \cite[p. 680]{LSV}  to $g=\psi$ and 
the zero-average function
 $f=\rho_\alpha-\mathbf 1
\in \CC_{*,1}+\real$, we have, if
$\alpha>0$,
\begin{equation}\label{00}
\left|\int_0^1 \psi \circ T_\alpha^k \, dx \right|=
\left| \int_0^1 \psi \LL_\alpha^k(\rho_\alpha-1)\, dx\right | \le C_\alpha \|\psi\|_{L^\infty} \frac{(\log k) ^{1/\alpha}}{k^{-1+1/\alpha}}\, ,
\end{equation}
and, for any $\beta >0$,
\begin{equation}\label{00'}
\left|\int_0^1 \psi \circ T_\beta^k \, dx -
\int_0^1 \psi \, d\mu_\beta \right|\le C_\beta \|\psi\|_{L^\infty} \frac{(\log k) ^{1/\beta}}{ k^{-1+1/\beta}}\, .
\end{equation}
If $\alpha=0$, then the spectral gap of $\LL_0$ on $\CC^1$ e.g. gives
a constant $C\ge 1$ so that
\begin{equation}\label{00''}
\left|\int_0^1 \psi \circ T_0^k \, dx \right|\le C \|\psi\|_{L^\infty} 
2^{-k} \, .
\end{equation}
Taking $k$ large enough, depending on $\beta$ and $\alpha$, the three expressions
\eqref{00}--\eqref{00'}--\eqref{00''}
are thus $o(\beta-\alpha)$. 
More precisely,  fixing $\xi >0$, there 
is $C$ so that,  for all 
$$k > C(C_{\max(\alpha,\beta)} (\beta-\alpha)^{-(1+\xi)})^{1/(-1+1/\max(\alpha,\beta))}$$
we have
\begin{equation}\label{both}
\left|\int_0^1 \psi \circ T_\alpha^k \, dx \right|
+ \left|\int_0^1 \psi \circ T_\beta^k \, dx -
\int_0^1 \psi \, d\mu_\beta \right|
\le C\|\psi\|_{L^\infty}(\beta-\alpha)^{1+\xi}\, .
\end{equation}
Letting ${\bf 1}$
be the constant function $\equiv 1$, it thus suffices to bound 
\begin{align}
\nonumber \frac{1}{\beta-\alpha}
\biggl(\int_0^1 \psi \circ T_\beta^k \, dx-
\int_0^1 \psi \circ T_\alpha^k \, dx\biggr)&=\frac{1}{\beta-\alpha}
\int_0^1 \psi (\LL_\beta^k {\bf 1 }-   \LL_\alpha^k {\bf 1}) \, dx\\
\label{calc00}&
=
 \sum_{j=0}^{k-1} \int_0^1 \psi \LL_\beta^{j}
 \biggl ( \frac{ \LL_\beta-\LL_\alpha}
{\beta-\alpha}
( \LL_\alpha^{k-j-1}({\bf 1})) \biggr )\, dx\, , 
\end{align} 
uniformly in $\beta \to \alpha$.
For this, we shall use below that for any $\varphi \in \CC^2(0,1]$, any $x\ne 0$, and
any $\beta \ne \alpha$,
\begin{equation}\label{zero}
\frac{\LL_\beta \varphi(x)-\LL_\alpha\varphi(x)}{\beta-\alpha}
=\partial_\alpha \LL_\alpha \varphi(x) + 
\frac{1}{\beta-\alpha} \int_\alpha^{\beta} \partial^2_{\gamma} \LL_\gamma \varphi(x)
(\gamma-\alpha)\, d \gamma \, .
\end{equation}
It is easy to check  (see the proof of \cite[Thm 2.2]{B3}) that we have
\begin{equation}\label{partg}
\partial_\alpha g_\alpha (x)=-X_\alpha(x) \NN_\alpha (1/T'_\alpha)  (x)=
-\frac{X_\alpha(x)}{T'_\alpha(g_\alpha(x))^2}\, , 
\end{equation}
and, more generally, for any $\varphi \in \CC^1(0,1]$ and any $x \ne 0$, we have
\begin{equation}\label{one}
 \partial_\alpha \LL_\alpha  (\varphi)(x)=\partial_\alpha \NN_\alpha(\varphi)(x)= \MM_\alpha (\varphi)(x)\, ,
\end{equation}
where  we set for $x\ne 0$
\begin{align}\label{two}
\MM_\alpha (\varphi)(x) &= -X'_\alpha\NN_\alpha (\varphi)(x)
-X_\alpha \NN_\alpha (\varphi'/T'_\alpha)  (x)
+ X_\alpha \NN_\alpha (\varphi T''_\alpha /(T'_\alpha)^2) (x)\\
\nonumber
&= -(X_\alpha \cdot  \NN_\alpha(\varphi))' (x).
\end{align}
Using \eqref{one} and \eqref{two} (twice), we also get, for $x \ne 0$ and $\varphi \in \CC^2(0,1]$, 
\begin{align}
\nonumber &\partial^2_{\alpha} \LL_\alpha (\varphi)(x)=
-\partial_\alpha\left((X_\alpha \NN_\alpha( \varphi))'\right)(x)\\
\nonumber& \qquad \quad=
-\left((\partial_\alpha X_ \alpha') (\NN_\alpha (\varphi) )\right)(x)-
\left(X_\alpha' (\partial_\alpha \NN_\alpha( \varphi))\right)(x)\\
\nonumber&\qquad\qquad\qquad\qquad\qquad- \left(\partial_\alpha X_\alpha ( \NN_\alpha( \varphi))'\right)(x)- \left( X_\alpha \partial_\alpha( \NN_\alpha( \varphi))'\right)(x)\\
\label{magic}&\qquad\quad  =-\left((\partial_\alpha X_ \alpha) (\NN_\alpha( \varphi))\right)'(x)+
X'_\alpha \left(X_\alpha \NN_\alpha(\varphi) \right)'(x) 
+ X_ \alpha \left(X_\alpha \NN_\alpha (\varphi)\right)''(x)\, .
\end{align}

Returning to \eqref{calc00}, we assume that $\beta > \alpha>0$.
For $k\ge 1$,  we get, using \eqref{zero}--\eqref{two},
\begin{align}
\nonumber  
 \sum_{j=0}^{k-1} \int_0^1 \psi \LL_\beta^{j} \biggl ( \frac{ \LL_\beta-\LL_\alpha}
{\beta-\alpha}
 (\LL_\alpha^{k-j-1}({\bf 1})) \biggr )\, dx&=-
\sum_{j=0}^{k-1} \int_0^1 \psi \LL_\beta^{j} \left(\left[X_\alpha \NN_{\alpha}( \LL_\alpha^{k-j-1}({\bf 1}))\right] '\right)\, dx
\\
\label{calc0} &\, +
\int_{\alpha}^{\beta} \frac{\gamma-\alpha}{\beta- \alpha}
\sum_{j=0}^{k-1} \int_0^1 \psi \LL_\beta^{j} \left[\partial^2_\gamma \LL_\gamma( \LL_\alpha^{k-j-1}({\bf 1})) \right]
\, dx d \gamma\, .
\end{align}
Consider the first term in the right-hand side of \eqref{calc0}.
Observe  that  $\LL_\alpha {\bf 1}\in \CC_{*,1}(\alpha,a,b_1) \cap \CC_2$,
 so that,
recalling Proposition~\ref{cheat'}, we have that  $\LL_\alpha^{k-j-1}({\bf 1})$ is
in  $\CC_{*,1}(\alpha) \cap \CC_2$ and thus in
$\CC_{*,1}(\gamma)\cap \CC_2$ for any
$\gamma \ge \alpha$ up to increasing $a$ uniformly in
$j$ and $k>j-1$. Note that
$|(\NN_\alpha( \LL_\alpha^{k-j-1}({\bf 1})))'(x)|\le b_1c_2 x^{-1-\alpha}$.
Proceeding as in Step~ 0 (using \eqref{Cstarii} to
invoke \eqref{half}  to get \eqref{half'}), we obtain
a constant $C>1$ so that
$|[X_\alpha \NN_{\alpha}( \LL_\alpha^{k-j-1}({\bf 1}))] '(x)|\le  C(|\log x|+1)$ 
for all $1\le j \le k-1$
(in particular $\sup_k \sup_{1\le j \le k-1}\|[X_\alpha \NN_{\alpha}( \LL_\alpha^{k-j-1}({\bf 1}))] '\|_1<\infty$).

Next, if
$0<\alpha<\beta <1$,
by using 
 \cite[Thm 2.4.14]{Goth} (as in Step 0),  we get summability of the first term of the
expression in the right-hand side of \eqref{calc0} as $k\to \infty$: 
\begin{equation}\label{more}
\left|\sum_{j=0}^{k-1} \int_0^1 \psi \LL_\beta^{j} ([X_\alpha \NN_{\alpha}( \LL_\alpha^{k-j-1}({\bf 1}))] ')\, dx
\right| \le 
C_\beta \|\psi\|_{L^\infty} \sum_{j=0}^{k-1}  
 \frac{1} {j^{(1/\beta)-\epsilon}}  \, .
\end{equation}

Finally, we consider the second
term of the right-hand side 
of \eqref{calc0}.  (This is where we require the derivatives
of order two in $\CC_2$.) Applying \eqref{magic}
to $\varphi=\LL_\alpha^{k-j-1}(\mathbf{1})\in \CC_2\cap \CC_{*,1}(\alpha)$, and using Proposition ~\ref{cheat'},
we see  that for any $\alpha \le \gamma\le \beta$, 
\begin{equation}\label{H1}
|[\partial^2_{\gamma} \LL_{\gamma}( \LL_\alpha^{k-j-1}({\bf 1}))] (x)|\le  C(|\log x|+1)^2
\, .
\end{equation}
 Indeed,  recalling \eqref{Xb} and \eqref{partg}, first note that 
$$|\partial_\alpha X_\alpha  (x)|\le c x^{1+\alpha} (|\log x|+1)^2\, ,\, 
\, |\partial_\alpha X_\alpha'|\le  c x^\alpha (|\log x|+1)^2\, ,\, \, 
|\partial_\alpha X_\alpha''|\le  c x^{\alpha-1} (|\log x|+1)^2\, .
$$
In addition,
for $\varphi$ in $\CC_{*,1}(\alpha)\cap \CC_2$, we have
$$ |(\NN_\alpha (\varphi))''(x)|\le  \frac{b_2}{ x^2} \rho_\alpha(x)m(\NN_\alpha(\varphi))
 \le \frac{b_2 c_2}{ x^{2+\alpha}} m( \varphi)\, , 
 \quad 
 |(X_\alpha \NN_\alpha (\varphi))''(x)| \le \frac{\hat c }{x} \, .
 $$
 
  Labelling the three terms from the right-hand side of \eqref{magic} as $I$, $II$, and 
  $III$, we expand them via the Leibniz equality, obtaining seven functions:
  \begin{align}
\nonumber &  I=-(\partial_\alpha X_\alpha)'(\NN_\alpha\varphi)+\partial_\alpha X_\alpha(\NN_\alpha\varphi)'\, ,\, \, \,\,\quad
II=(X_\alpha')^2 \NN_\alpha (\varphi)+ X_\alpha'X_\alpha (\NN_\alpha(\varphi))'\, , \\
\label{seven} &III=
X_\alpha [ X_\alpha'' \NN_\alpha(\varphi) + 2 X_\alpha' (\NN_\alpha(\varphi))'
+ X_\alpha (\NN_\alpha(\varphi))'']\, .
\end{align}

 By the above, since  $\varphi=\LL_\alpha^{k-j-1}(\mathbf{1})\in \CC_2\cap \CC_{*,1}(\alpha)$, we can bound $|I|$ by $$(c x^\alpha (|\log x|+1)^2)(cx^{-\alpha})+\left((cx^{1+\alpha})(|\log x|+1)^2\frac{b_1}x \frac{2ac_2}{x^\alpha}m(\varphi)\right)\, 
 .$$
 Similarly $|II|\le (cx^\alpha(|\log x|+1)(c(|\log x|+1))$ and $|III|\le (cx^{1+\alpha}(|\log x|+1))\hat c x^{-1}$. This proves \eqref{H1}.

Applying  \cite[Thm 2.4.14]{Goth}
once more (as in Step 0) we thus get the  bound
\begin{equation}\label{better}
\frac{\|\psi\|_{L^\infty}}{\beta- \alpha}
\int_{\alpha}^{\beta} (\gamma-\alpha) C_\beta
\sum_{j=0}^{k-1} 
\frac{1}{j^{(1/\beta)-\epsilon}})
\, d\gamma\le C_\beta \|\psi\|_{L^\infty} (\beta -\alpha)\, .
\end{equation}

If $\alpha \in (0,1)$, the  case $\beta<\alpha$  can be handled similarly, substituting
$$
\sum_j \LL_\beta^{j} ( \LL_\beta-\LL_\alpha)
 \LL_\alpha^{k-j-1}=-\sum_j\LL_\alpha^{j}  (\LL_\alpha-\LL_\beta)
 \LL_\beta^{k-j-1}$$ 
 in \eqref{calc00}, and replacing  \eqref{zero} by
 $\frac{\LL_\alpha \varphi(x)-\LL_\beta\varphi(x)}{\alpha-\beta}
=\partial_\beta \LL_\beta \varphi(x) + 
\frac{1}{\alpha-\beta} \int_\beta^{\alpha} \partial^2_{\gamma} \LL_\gamma \varphi(x)
(\gamma-\beta)\, d \gamma$. 

Finally, if $\alpha=0$, we use 
$
\sum_j \LL_\beta^{j} ( \LL_\beta-\LL_0)
 \LL_0^{k-j-1}=-\sum_j\LL_0^{j}  (\LL_0-\LL_\beta)
 \LL_\beta^{k-j-1}$, with \eqref{zero}, and exploit
 \eqref{toremember} and \eqref{furtherr}.

This proves that for   $\psi\in L^\infty$
the map $\beta \mapsto \int \psi(x)\rho_\beta(x)\, dx$
is locally Lipschitz on $[0,1)$.

\smallskip
{\bf Step 2:}
Still assuming that $\psi$ is bounded, we next prove that $\beta \mapsto \int \psi \rho_\beta \, dx$ is differentiable
at $\beta =\alpha\in [0,1)$  and check that the derivative
takes the announced value.
 To prove differentiability,  recalling \eqref{both} and
setting $k(\beta)=k(\alpha,\beta,\xi)=C(C_\beta (\beta-\alpha)^{-(1+\xi)})^{1/(-1+1/\max(\alpha,\beta))}$,  for
some small $\xi>0$, it suffices to check that
\begin{equation}\label{forbeta'}
\sum_{j=0}^{ k(\beta) }
\int_0^1 \psi \LL_\beta^{j} 
([X_\alpha \NN_{\alpha}( \LL_\alpha^{k(\beta)-j}({\bf 1}))] ')\, dx+
\int_{\alpha}^{\beta} \frac{\gamma-\alpha}{\beta- \alpha}
\sum_{j=0}^{k(\beta)-1} \int_0^1 \psi \LL_\beta^{j} [\partial^2_\gamma \LL_\gamma( \LL_\alpha^{k(\beta)-j-1}({\bf 1}))]
\, dx d \gamma\, ,
\end{equation}
converges, when $\beta \to \alpha>0$  or $\beta \downarrow \alpha=0$, to
\begin{equation}\label{candidate'}
\sum_{j=0}^{\infty} \int_0^1 \psi \LL_\alpha^{j} ([X_\alpha \NN_\alpha (\rho_\alpha)] ')\, dx=
 \int_0^1 \psi \sum_{j=0}^{\infty}\LL_\alpha^{j} ([X_\alpha \NN_\alpha(\rho_\alpha)] ')\, dx \, .
\end{equation}

By \eqref{better},  the second term in \eqref{forbeta'} converges to zero as
$\beta \to \alpha$. Next, for $\alpha \in [0,1)$,
fixing $\eta >0$ small,  by Step 0  we may take $K=K_\eta$ large enough so
that the $K$-tail of \eqref{candidate'} is $<\eta/4$, while the $K$-tail of the first term of
\eqref{forbeta'} is $<\eta/4$ uniformly in $\beta$. It thus suffices, for every
fixed $0\le j \le K$,
to show that the following difference tends to zero as
$\beta \to \alpha >0$ or $\beta \downarrow 0$:
\begin{equation}\label{forgot}
\int_0^1 \left[(\psi \circ T^j_\beta)
([X_\alpha \NN_{\alpha}( \LL_\alpha^{k(\beta)-j}({\bf 1}))] ')-
(\psi \circ T_\alpha^{j}) ([X_\alpha \NN_\alpha(\rho_\alpha)] ') \right]\, dx\, .
\end{equation}
 So it is sufficient to show that there exists
 $N_{\eta}\ge 1$ so that 
 \begin{equation}\label{first}
\|[X_{\alpha} \NN_{\alpha}( \LL_\alpha^{k}({\bf 1}))] '-
 [X_\alpha \NN_\alpha (\rho_\alpha)] ' \|_{L^1}<\frac{\eta}{2K} \, ,
 \, \,  \forall k \ge N_\eta\, .
\end{equation}
If $\alpha=0$, this is easy, since $\rho_\alpha={\bf 1}$ so that the expression \eqref{first}
vanishes trivially.

If $\alpha \in (0,1)$,
setting $\phi_k:= \LL_\alpha^{k}({\bf 1})$, we note that
 the bound on \cite[p. 680]{LSV} applied to $g={\bf 1}$ and $f={\bf 1}-\rho_\alpha$ implies
$
\|\phi_k -\rho_\alpha\|_1
\le C_ \alpha k^{1-1/\alpha}(\log k)^{1/\alpha}
$.
(This is not summable if $\alpha \ge 1/2$, but it does tend to
zero for all $\alpha \in (0,1)$.)
Therefore,   
\begin{equation}\label{trivial}
\|\NN_\alpha(\phi_k-\rho_\alpha)\|_1\le C_\alpha \|\phi_{k+1} -\rho_\alpha\|_1 \le C_\alpha  \frac{(\log k)^{1/\alpha} }{  k^{-1+1/\alpha}}\, . 
\end{equation}
Since $X_\alpha'\in L^\infty$, it thus suffices to show that 
$$\|
X_\alpha [\NN_{\alpha}( \LL_\alpha^{k}({\bf 1}))] '-
X_\alpha [ \NN_\alpha (\rho_\alpha)] ' \|_{L^1}<\frac{\eta}{2K}\, .
$$
For this, first observe that  Proposition~\ref{cheat'} implies
that there exists $b_1$  so that, for $\varphi=\rho_\alpha$ and $\varphi=\bf 1$,
$$
\left|[\NN_\alpha \LL^{k}_\alpha(\varphi)]'(x)\right|\le  \frac{b_1}{ x} \varphi(x)
\le \frac{b_1 ac_2}{x^{1+\alpha}}\, , \, 
\forall k \ge 1\, , \, \, \forall  x \in (0,1] \, .
$$
Since $|X_\alpha (z)|\le c_\alpha |z^{1+\alpha} (\log z+\log 2)|$,
it follows that there exist $C, \hat C>0$ so that for any $\bar x\in (0,1]$,
all $k\ge 0$, and,  for  $\varphi=\rho_\alpha$ and $\varphi=\bf 1$,
\begin{equation}\label{easy}
\int_0^{\bar x} |X_\alpha(z)(\NN_\alpha \LL^{k}_\alpha(\varphi))'(z)|\, dz \le 
C \int_0^{\bar x} (|\log z| + 1) \, dz \le \widehat C  \bar x (|\log \bar x| + 1)\, .
\end{equation}
We will choose $\bar x$ to be small and then choose $k$ large enough that the remaining integral is small.
We decompose the remaining integral, for  $\varphi=\rho_\alpha$ and $\varphi=\bf 1$,  as
$$
\int_{\bar x}^1 X_\alpha(z)(\NN_\alpha \LL^{k}_\alpha(\varphi))'(z) dz=
\int_{\bar x}^1 X_\alpha(z)(\LL^{k+1}_\alpha(\varphi))'(z) dz
+\int_{\bar x}^1 X_\alpha(z)\frac{(\LL^{k}_\alpha(\varphi))'((z+1)/2)}{2} dz  .
$$
We focus on the first term above, the estimate for the second one being easier.
We shall set $\bar x=x_\ell$, for suitable $\ell \ge 1$ to be determined below, where
\begin{equation}\label{upper}
x_\ell:=(g_\alpha)^\ell (1)\le   2^{1/\alpha^2+1/\alpha} \ell^{-1/\alpha} 
\end{equation}
by \cite[Lemma 3.2]{LSV}.
Next,   for every $1\le m \le k$,  setting
$y_m(x_\ell)=(g_\alpha)^m(x_\ell)=x_{\ell+m}$, and
$\phi= \phi_{k+1-m}-\rho_\alpha$, we have
\begin{align*}
\left\|\chi_{x> x_\ell} [\LL^m_\alpha(\phi)]'\right\|_1
&\le
\left\| \LL_\alpha^m(\chi_{y > y_m} |\phi'|/(T^m_\alpha)')\right\|_1+
\left\| \LL_\alpha^m(\chi_{y > y_m}|\phi| |(T^m_\alpha)''|/(T^m_\alpha)')^2)\right\|_1\\
&\le
\left\| \chi_{y > y_m}|\phi'| \cdot |(T^m_\alpha)'|^{-1} \right\|_1+
\left\| \chi_{y > y_m}  |\phi| \left| (T^m_\alpha)'' \right| \left|(T^m_\alpha)'\right|^{-2} \right\|_1\, .
\end{align*}
There exist $\lambda_m=\lambda_m(x_\ell)<1$ and $\Lambda_m(x_\ell) <  \infty$ 
(both depending on $\alpha$)  so that the first term in the right-hand side is 
$<\lambda_m \| \chi_{y > y_m(\ell)}|\phi'| \|_1$
and the  second term is $\le \Lambda_m \|\phi\|_1$.
In fact,  we claim that there exists $C_\alpha>0$
so that for all $\ell$
$$\lambda_m(x_\ell) \le C_\alpha  (1+m/\ell)^{-1-1/\alpha}\, .
$$
Indeed, recalling that $f_\alpha=T_\alpha|_{[0,1/2]}$,
we have $\lambda_m(x_\ell)^{-1}=(f_\alpha^m)'(y)$ for
some $y\ge y_m(x_\ell)$, and  bounded distortion
\footnote{See e.g. \cite[(2) p. 678]{LSV} for the bounded
distortion property.} of $f_\alpha^m$ on
$(y_m,f_\alpha(y_m))=(y_m, y_{m-1})$
gives 
\begin{equation}\label{llj}
\lambda_m(x_\ell)\le C \frac{f_\alpha(y_m)-y_m}{f_\alpha(x_\ell)-x_\ell}
= C \frac{ y_m^{1+\alpha}}{x_\ell^{1+\alpha}}\le C_\alpha \frac{1}{(1+m/\ell)^{1+1/\alpha}}
\, ,
\end{equation}
where we used the upper bound $
y_m (x_\ell)=x_{\ell+m} \le   2^{1/\alpha^2+1/\alpha} (\ell+m)^{-1/\alpha}
$
from \eqref{upper} and the lower bound from
\footnote{Note that $\alpha+1$ should read $\alpha-1$ 
in line 7 of the proof of \cite[Prop. 2, p 606]{BT}, that
$+\alpha(\alpha+1)/(2u_{n+1})$ should be replaced by $-\alpha(\alpha-1)/(2u_{n+1})$
in line 8, that $+\log n \cdot \alpha(\alpha+1)/2$ should be replaced
by $-\log ((1+\alpha n)/(1+\alpha)) \cdot (\alpha-1)/2$ on line 10, and 
that $\log n \cdot \alpha(\alpha+1)/(2n)$
should be replaced
by $-\log  (1+\alpha n) \cdot (\alpha-1)/(2n)$ in line 12.} \cite[p. 606]{BT}
(replacing their $x+x^{1+\alpha}$ by our $x+2^\alpha x^{1+\alpha}$)
\begin{equation}\label{lowery}
x_\ell
\ge c   (2^\alpha \alpha)^{-\frac1\alpha}\ell^{-1/\alpha} \, .
\end{equation}

Recalling that $\phi=\phi_{k+1-m} -\rho_\alpha$, we get, if $\alpha \in (0,1)$,
\begin{align*}
&\| \chi_{x> x_\ell} [\LL^m_\alpha(\phi)]'\|_1
\le \lambda_m (x_\ell)
\| \chi_{y > y_m}  (| \phi_{k+1-m}'|+| \rho_\alpha'|) \|_1+ 
\Lambda_m  C_\alpha
\frac{(\log (k+1-m))^{1/\alpha}}{ (k+1-m)^{1/\alpha-1}} \, .
\end{align*}
Recall that $|\phi_\ell'(x)| \le (a/x) \phi_\ell \le C b_1 x^{-\alpha-1}$
so that $\| \chi_{y > y_m} | \phi_\ell'| \|_1  \le C y_m^{-\alpha}$
and $|\rho_\alpha'(x)| \le  c_2 b_1 x^{-\alpha-1}$, giving the same
asymptotics, and note that \eqref{lowery} 
gives 
$$y_m (x_\ell)
\ge c   \alpha^{-\frac1\alpha}(m+\ell)^{-1/\alpha}$$
uniformly in $k$. 
Hence, using \eqref{lowery} for $y_m(\ell)=x_{m+\ell}$, we get 
$$\| \chi_{y > y_m(x_\ell)}  (| \phi_{k+1-m}'|+| \rho_\alpha'|) \|_1
\le C \int_{y_m}^1 y^{-1-\alpha} \, dy\le C y_m^{-\alpha} \le C \alpha (m+\ell)\, .
$$ 
Clearly, \eqref{llj} implies
\begin{equation}\label{tricky}
C \alpha (m+\ell) \lambda_m(x_\ell) \le C_\alpha \alpha \frac{m+\ell }{(1+m/\ell)^{1+1/\alpha}}
\le C_\alpha \alpha \frac{\ell }{(1+m/\ell)^{1/\alpha}}
\, .
\end{equation}
Choosing first $\ell\ge 1$ to make \eqref{easy} small, then  $m\ge \ell$ to make
\eqref{tricky} small, and finally
taking $k\ge m$ large enough (i.e., $\beta$
close enough to $\alpha$) so that 
$$\Lambda_m C_\alpha (\log (k+1-m))^{1/\alpha}(k+1-m)^{1-1/\alpha}$$
is small, proves \eqref{first} in view of \eqref{trivial} and \eqref{easy}.

This proves the result for bounded $\psi$. If $\|\psi\|_{ L^q}=1$ for 
$(1-\alpha)^{-1}<q<\infty$, we 
observe that $\mathrm{Leb} (\{\psi(x)>M\})\le M^{-q}$ and define
 $\psi_M(x)=\min(M,\psi(x))$, noting that $\|\psi-\psi_M\|_{L^1}\le M^{1-q}$ and
more generally $\|\psi-\psi_M\|_{L^r}\le M^{1-q/r}$ 
for $r>1$ close to $1$. 
Since $|\log (2x)|\in L^{r/(r-1)}(dx)$ for all $r>1$, we can generalise Steps 0 and 1
to $\psi\in L^q(dx)$
the result by  taking $\epsilon>0$
very small and  $r>1$ so that $q>r (1-\epsilon\alpha)/(1-\alpha-\epsilon\alpha)$ 
and choosing $\eta>0$ small enough so that
$
\frac{1}{(q/r)-1}<\eta< \frac{1}{\alpha}-\epsilon-1
$,  
taking $M(j)= j^\eta$, and 
decomposing $\psi=(\psi-\psi_{M(j)})+
\psi_{M(j)}$ in the $j$th term of \eqref{parts},  \eqref{more}, and \eqref{calc0}.
We get two series each time. The first one is convergent because $\eta((q/r)-1)>1$ while the second
one converges because $\eta+1<1/\alpha -\epsilon$. 
For Step~2, we take $M(k)=k^\eta$ for $\eta\in (0, 1/\alpha-1)$ in \eqref{forgot}.

Finally, the  claim about continuity of the derivative follows from the linear response formula and our control
on the tails of the absolutely convergent series therein.
\end{proof}

\appendix

\section{Proof of  Propositions ~\ref{cheat'} and ~\ref{replaces4.10}}
\label{theproof}

\begin{proof}[Proof of Proposition ~\ref{cheat'}]
The proof  for the first derivative
is similar to that of the proof of \cite[Lemma 2.3, Lemma 5.1]{LSV}.
We concentrate on the statement for $\NN_\alpha$, the proof for
$\LL_\alpha$ follows easily since $(\LL_\alpha-\NN_\alpha)(\varphi)(x)=
\varphi((x+1)/2)/2$.
Let $0\le \alpha <1$.
We have 
$T_\alpha'(x)= 1+ 2^ \alpha(\alpha+1) x^\alpha\ge 1$,
 $T_\alpha''(x)= 2^\alpha (\alpha+1)\alpha x^{\alpha-1} \ge 0$, and
 $T_\alpha'''(x)= 2^\alpha (\alpha+1)\alpha  (\alpha-1) x^{\alpha-2} \le 0$. 
    Throughout this proof, we set (recall \eqref{invv}) $
y=g_\alpha(x)
$.

For  $\varphi$ as in the statement of the proposition, we have (both terms are positive)
 \begin{align*}
- (\NN_\alpha \varphi)'(x)&=
 \frac{T_\alpha''(y) }{(T_\alpha'(y))^3} \varphi(y)- \frac{1}{(T_\alpha'(y))^2} \varphi'(y) \\
 &\le \biggl (\frac{T''_\alpha(y) }{(T_\alpha'(y))^2} + \frac{b_1}{y(T_\alpha'(y))}\biggr )
 \frac{\varphi(y)}{T_\alpha'(y)}\\
& \le \frac{b_1}{x}(\NN_\alpha \varphi)(x)\sup_{y \in [0,1/2]}
\biggl [\frac{ T_\alpha(y)}{b_1}
\cdot \biggl (\frac{ T''_\alpha(y)}{(T_\alpha'(y))^2} + 
 \frac{b_1}{y T_\alpha'(y)}\biggr )\biggr ] \, .
 \end{align*}
 Let $\Omega_1(y)$ be the term in square brackets, then we find if $b_1\ge 1+\alpha$
 \begin{align}
\nonumber \Omega_1(y)&=
\frac{T_\alpha(y)}{ b_1 y T'_\alpha(y)}\biggl (
\frac{yT''_\alpha(y)} {T'_\alpha(y)}+b_1\biggr )
\le \frac{ 1+2^\alpha y^\alpha}{1+2^\alpha (1+\alpha)y^\alpha}
\cdot \biggl (\frac{1}{b_1}\frac{  2^\alpha (\alpha+1)\alpha y^{\alpha}}{1+2^\alpha (1+\alpha)y^\alpha } + 1\biggr )\\
\nonumber & =
\biggl (1- \frac{ 2^\alpha \alpha y^\alpha}{1+2^\alpha (1+\alpha)y^\alpha}\biggr ) 
\cdot 
\biggl (1+\frac{1}{b_1}\frac{  2^\alpha (\alpha+1)\alpha y^{\alpha}}{1+2^\alpha (1+\alpha)y^\alpha } \biggr )\\
\label{controlb1} &\le \biggl (1- \frac{ 2^\alpha \alpha y^\alpha}{1+2^\alpha (1+\alpha )y^\alpha}\biggr ) 
\cdot 
\biggl (1+\frac{  2^\alpha \alpha y^{\alpha}}{1+2^\alpha (1+\alpha)y^\alpha } \biggr )\, ,
 \end{align}
 which is $\le 1$ for all $y \in [0,1/2]$  (we used
$(\alpha+1)/b_1\le 1$ in the last line).
 Note that if $\alpha=0$ then
 $\Omega_1(y)=\frac{ 2}{b_1}
 \frac{b_1}{2  }\equiv 1
 $.
To get the reverse inequality,  observe that
$\overline \Omega_1(y)=
\frac{T_\alpha(y)}{  y T'_\alpha(y)}\bigl (
\frac{yT''_\alpha(y)} {\bar b_1 T'_\alpha(y)}+ 1\bigr ) \ge 1
$
if $\bar b_1>0$ is small enough (just revisit \eqref{controlb1}).
  
Next, writing $T$ instead of $T_\alpha$ for simplicity
$|(\NN_\alpha \varphi)''(x)|$ is bounded by  (all  terms below are nonnegative)
\begin{align*}
&
 -3 \frac{\varphi'(y)T''(y)}{(T'(y))^4}-
\frac{\varphi (y)T'''(y)}{(T'(y))^4}
+3\frac{\varphi(y)(T''(y))^2}{(T'(y))^5}
+ \frac{\varphi''(y)}{(T'(y))^3}
\\
&\qquad \le \NN_\alpha\varphi(x)\bigg( 
3 \left(\frac{b_1}y\right)  \frac{T''(y)}{(T'(y))^3}-\frac{T'''(y)}{(T'(y))^3}
+3\frac{(T''(y))^2}{T'(y))^4}+ \left(\frac{b_2}{y^2}\right)\frac1{(T'(y))^2}\bigg)\\
&\qquad\le \frac{b_2}{x^2}\NN_\alpha\varphi(x)
\bigg[\frac{T(y)^2}{b_2|T'(y)|^2}\bigg( 3\left(\frac{b_1}y\right) \frac{2^\alpha(\alpha+1)\alpha  y^{\alpha-1}}{T'(y)}
+\frac{2^\alpha (\alpha+1)\alpha|\alpha-1|y^{\alpha-2}}{T'(y)}
\\
&\qquad\qquad\qquad\qquad \qquad\qquad\qquad\qquad\qquad\qquad\qquad
+3
\frac{(2 ^\alpha (\alpha+1)\alpha y^{\alpha-1})^2}{(T'(y))^2}
+\left(\frac{b_2}{y^2}\right)\bigg)\bigg]\, .
\end{align*}
The term $\Omega_2(y)$ in square brackets can be written
\begin{align*}
&\frac{T(y)^2}{y^2|T'(y)|^2}\bigg(  1 
+ \frac{2^\alpha \alpha y^\alpha}{(1+2^\alpha(\alpha+1)y^\alpha)}
\frac{1}{b_2}
\biggl [3 b_1  (\alpha+1)
+ (1-\alpha^2)
+3
\frac{2 ^\alpha (\alpha+1)^2\alpha y^{\alpha}}{1+2^\alpha(\alpha+1)y^\alpha}\biggr ]\bigg) \, .
\end{align*}
We can fix $b_2>3b_1(1+\alpha)+20$ large enough so that  $\Omega_2(y)\le 1$ for all  $y \in[0,1/2]$
because
\begin{align*}
& \left(\frac{T(y)}{y T'(y)}\right)^2= \left(\frac{1+ 2^\alpha y^{\alpha}}{1+ 2^\alpha (\alpha+1)y^\alpha}\right)^2 =
\left(
1- \frac{2^\alpha \alpha y^{\alpha}}{1+2^\alpha (\alpha+1) y^\alpha}\right)^2\, .
\end{align*}
 (Note that if $\alpha=0$ then
 $\Omega_2(y)\equiv 1 $.)
For the reverse inequality,  observe that
\begin{align*}
\overline \Omega_2(y):=&
\left(
1- \frac{2^\alpha \alpha y^{\alpha}}{1+2^\alpha (\alpha+1) y^\alpha}\right)^2\\
&\cdot 
\bigg(  1 
+ \frac{2^\alpha \alpha y^\alpha}{(1+2^\alpha(\alpha+1)y^\alpha)}
\frac{1}{\bar b_2}
\biggl [3 \bar b_1  (\alpha+1)
+ (1-\alpha^2)
+3
\frac{2 ^\alpha (\alpha+1)^2\alpha y^{\alpha}}
{1+2^\alpha(\alpha+1)y^\alpha}\biggr ]\bigg)\ge 1 \, ,
\end{align*}
if  $\bar b_1/\bar b_2>0$ is large
enough.
  \end{proof}

\begin{proof}[Proof of Proposition~\ref{replaces4.10}]
Note that $m(\varphi)=m(\LL_\alpha(\varphi))$ and $\LL_\alpha(\rho_\alpha)=\rho_\alpha$.
Also, $m(\varphi) \le 2 m(\NN_\alpha(\varphi))$ (using our assumption)
and  $\NN_\alpha(\rho_\alpha) < \rho_\alpha$.
By \eqref{controlb1} 
 we have $\LL_\alpha(\CC_{*,1}(\alpha))\subset \CC_{*,1}(\alpha, a , b_1)$
and  $\NN_\alpha(\CC_{*,1}(\alpha))\subset \CC_{*,1}(\alpha, 2a, b_1)$.

For the decay claim, 
use that $\CC_*(0)\subset \CC_*(\beta)$ for any $\beta \in (0,1)$, and
fix $\beta=\alpha$. 
For $\epsilon>0$ we set
\footnote{Identifying $[0,1]$ with the circle $\SS^1$.}  $\AAA_\epsilon \varphi(x)=(2\epsilon)^{-1}
\int_{y \in \SS^1: |x-y|<\epsilon} \varphi(y)\, dy$. Then, revisiting 
\cite[Proposition 3.3]{LSV} for $\alpha=0$, we see that we can take $n_\epsilon$ there to be
$|\log \epsilon|/\log 2$.
Since \cite[Lemma 3.1]{LSV}  implies $
\|\LL_0^{n_\epsilon}(\id-\AAA_\epsilon)(\varphi)\|_1\le \frac{18 a b_1 c_2}{\beta(1-\beta)}\|\varphi\|_1$
if $\varphi \in \CC_{*,1}(\beta,b_1)$,  and since
$\CC_*(\alpha)$ is invariant under $\LL_0$, the
 first paragraph of the
proof of \cite[Thm~ 4.1]{LSV}, taking
$\epsilon=n^{-1/\alpha}$ gives \eqref{toremember}.
(Note that \cite[Lemma 2.4]{LSV} is not needed 
when invoking  \cite[Prop 3.3]{LSV} in the proof of 
\cite[Thm 4.1]{LSV} for $T_0$, since we may obtain an easier lower bound.)
\end{proof}


\section{A cone-only proof for $\alpha \in [0,1/2)$ and $\psi\in L^\infty$}
\label{coneonly}

We show how to modify the proof of Theorem ~\ref{main} to bypass the use
of Gou\"ezel's results \cite{Go, Goth} when $\alpha <1/2$ 
and $\psi\in L^\infty$ (exploiting only \cite{LSV}).

We first note that in the setting of
Proposition~\ref{replaces4.10} there exists $C_\alpha=C_\alpha(a,b_1)>0$, with
$\sup_{\beta \in[ \alpha,(1+\alpha)/2]} C_\beta(a,b_1)<\infty$, so that
for any  $\psi \in L^\infty$ and $\varphi\in \CC_{*,1}(\alpha)+ \real$ with
$\int \varphi \, dx=0$, 
\begin{equation}\label{toremember1}
\int_0^1 \psi \LL_\alpha^k(\varphi)\, dx \le 
C_\alpha \|\psi\|_{L^\infty}  \| \varphi\|_{L^1}\frac{(\log k)^{1/\alpha}}
{k^{\frac{1}{\alpha}-1} }\, , \, \,
\forall k \ge 1 \, ,
\end{equation}
Indeed,  note that \cite[Lemma 3.1]{LSV} applies to
$\CC_{*,1}$ instead of $\CC_*$, up to replacing $10a$ there by $18 a b_1 c_2$.
(In the computation, 
just use that $|\varphi(x)-\varphi(y)|\le \sup_{z\in [x,y]}
|\varphi'(z)|\epsilon\le 2ab_1c_2\epsilon x^{-1-\alpha}$, 
if $|x-y|\le \epsilon$ with $x\le y$.  The original $10$ in \cite{LSV} is obtained as $4\times 2+2$: 
the definition of our cone $\CC_{*,1}$ incorporates an additional factor of $2$, to make $4\times 2\times 2+2=18$ as well as introducing extra factor of $c_2$, as in \eqref{Cstar1_basic}, while the $b_1$ appears since we use the derivative of $\phi$, as just noted.
Finally, apply the argument \footnote{There is a typo there and one should take in fact $\epsilon=
n^{-1/\alpha}( 2^{2+1/\alpha} \gamma^{-1} (\frac{1}{\alpha}-1) \log n)^{1/\alpha}$.}
in the first paragraph of 
\footnote{Just like for \cite[Prop. 5.4]{LSV}.}
the proof of \cite[Thm 4.1]{LSV}.
The proof gives
$
C_\alpha= \frac{36  a b_1c_2}{\alpha(1-\alpha)} 
2^{(2+1/\alpha)(\frac{1}{\alpha}-1)}
\left ( \frac{\frac{1}{\alpha}-1}{\gamma}\right )^{1/\alpha}$, for some small $\gamma>0$.
In particular, $C_\alpha$ becomes very large as $\alpha\to 0$.
This ends the proof of \eqref{toremember1}.

\smallskip

We need to introduce the following cone:
\begin{align*}
\CC_3=\Bigg\{ \varphi \in \CC^3(0,1]\, \mid \, 
&  \varphi \in \CC_2\, ,\, \,\, 
 |\varphi'''(x)|\le \frac{b_3}{x^3}  \varphi(x)\, ,\,  \,\, \forall x \in (0,1] \Bigg\}\, .
\end{align*}
If $b_3\ge b_1$ is large enough then the invariance statements of Proposition~\ref{cheat'} also hold for $\CC_3$, indeed, noting that
 $T_\alpha^{(iv)}(x)= 2^\alpha (\alpha+1)\alpha  (\alpha-1)(\alpha-2) x^{\alpha-3}\ge 0$, we have
 \begin{align*}
&|(\NN_\alpha \varphi)'''(x)| 
\le \frac{|\varphi'''(y)|}{|T'(y)|^4}
+6\frac{|\varphi''(y)T''(y)|}{|T'(y)|^5}
+4\frac{|\varphi'(y)T'''(y)|}{|T'(y)|^5}
+15\frac{|\varphi'(y)(T''(y))^2|}{|T'(y)|^6}\\
&\qquad+\frac{|\varphi(y)T^{(iv)}(y)|}{|T'(y)|^5}
+4\frac{|\varphi(y)T'''(y)|}{|T'(y)|^6}
+6\frac{|\varphi(y)T''(y) T'''(y)|}{|T'(y)|^6}
+15\frac{|\varphi(y)(T''(y))^3|}{|T'(y)|^7} \, , \\
\end{align*}
and this is bounded by
\begin{align*}
& \NN_\alpha\varphi(x)\bigg(
\frac{b_3}{y^3} \frac1{|T'(y)|^3}
+\frac{6 b_2}{y^2}\frac1{|T'(y)|^4}
+ \frac{b_1}y
\biggl (4\frac{|T'''(y)|}{|T'(y)|^4}+ 15 \frac{|(T''(y))^2|}{|T'(y)|^5}\biggr )
\\
&\qquad\qquad\qquad+\frac{|T^{(iv)}(y)|}{|T'(y)|^4}
+4\frac{|T'''(y)|^2}{|T'(y)|^5}+ 6 \frac{|T''(y) T'''(y)|}{|T'(y)|^5}+ 15
\frac{|(T''(y))^3|}{|T'(y)|^6}\bigg)\\
&\quad \le \frac{b_3}{x^3}\NN_\alpha\varphi(x)\bigg[
\frac{T(y)^3}{ |T'(y)|^3 y^3}
\bigg (
1
+\frac{6 b_2 y}{b_3|T'(y)|}
+ \frac{b_1 y^2}{b_3}
\biggl (4\frac{|T'''(y)|}{|T'(y)|}+ 15 \frac{|(T''(y))^2|}{|T'(y)|^2}\biggr )
\\
&\qquad\qquad\qquad+
\frac{y^3}{b_3}\biggl (\frac{|T^{(iv)}(y)|}{|T'(y)|}
+4\frac{|T'''(y)|^2}{|T'(y)|^2}+ 6 \frac{|T''(y) T'''(y)|}{|T'(y)|^2}+ 15
\frac{|(T''(y))^3|}{|T'(y)|^3}\biggr )
\bigg ) \bigg]\, .
\end{align*}
The term $\Omega_3(y)$ in square brackets is  $\le 1$
for all $y\in[0,1/2]$ if $b_3$ is large enough because
\begin{align*}
& \left(\frac{T(y)}{y T'(y)}\right)^3= \left(\frac{1+ 2^\alpha y^{\alpha}}{1+ 2^\alpha (\alpha+1)y^\alpha}\right)^3=
\left(
1- \frac{2^\alpha \alpha y^{\alpha}}{1+2^\alpha (\alpha+1) y^\alpha}\right)^3\, .
\end{align*}

It is easy to see that $\rho_\alpha \in \CC_3$,
that $\LL_\alpha {\bf 1}\in \CC_{*,1}(\alpha,a,b_1) \cap \CC_3$, etc. In fact, each
occurrence of $\CC_2$ in the proof of Theorem~\ref{main} can be replaced by $\CC_3$.

We now go over the changes in the proof of Theorem~\ref{main}. Consider first Step~0:
If $\alpha \in (0,1/2)$ then,  using \eqref{useful}, \eqref{Xb}, \eqref{X'b},  and \eqref{X''b}, together
with \eqref{half} and \eqref{half'},
and the fact that $|(\NN_\alpha(\rho_\alpha))''|\le  ac_2 b_2 x^{-2-\alpha}$,
it is easy to see that  there exists $a>0$ and a
uniformly bounded constant $C_X>0$  
 so that $-(X_\alpha' \NN_\alpha( \rho_\alpha))'+C_X$  belongs
 to $\CC_{*,1}(\alpha, a,b_1)$, up to increasing $a$ and $b_1$
(see \eqref{Cstar1}).
Next \eqref{toremember1}   applied to
the zero-average function
 $\varphi=-(X_\alpha \NN_\alpha( \rho_\alpha))'
\in \CC_{*,1}+\real$ gives $C_\alpha>0$
so that
the $j$th term in the right-hand side of \eqref{parts}
is bounded by 
$$C_\alpha j^{1-1/\alpha}(\log j) ^{1/\alpha} 
\|\psi\|_{L^\infty}  \| (X_\alpha \NN_\alpha(\rho_\alpha))' \|_1\, .
$$
Since  $\alpha <1/2$, this is summable.

In Step~1, 
proceeding as in Step~ 0 in \S \ref{theprooft} (using \eqref{Cstarii} to
invoke \eqref{half} in order to get \eqref{half'}), we  find  for any $1\le j \le k-1$
 a real constant $C_{j,k}<\infty$ so that
\begin{equation}\label{coneclaim}
\{\pm X_\alpha' \NN_{\alpha}( \LL_\alpha^{k-j-1}({\bf 1}))+C_{j,k}\, , \,\, 
X_\alpha [\NN_{\alpha}( \LL_\alpha^{k-j-1}({\bf 1}))]' \}\subset
\CC_{1,*}(\alpha,2 a,B_1)\ \, .
\end{equation}
Indeed, to show \eqref{coneclaim}, setting $\varphi= \LL_\alpha^{k-j-1}({\bf 1}))\in \CC_*$,
and noting that $\varphi\ge 0$ and $\varphi'\le 0$, so that
$\NN_\alpha \varphi \ge 0$ and $(\NN_\alpha \varphi)'\le 0$ so that $X_\alpha (\NN_\alpha \varphi)'\ge 0$
it is enough to check that
$$
|X''_\alpha \NN_\alpha(\varphi)|(x)+| X_\alpha' (\NN_\alpha (\varphi))' |(x) \le \frac{B_1}{2x}
|X'_\alpha(x)| \NN_\alpha(\varphi)(x)\,  , 
$$
(which follows from $-\NN_\alpha(\varphi)' (x)\le b_1 \NN_\alpha(\varphi)(x)/x$
and   \eqref{X'b}, \eqref{X''b}),
and
$$
|  X_\alpha' (\NN_\alpha (\varphi))' + X_\alpha(\NN_\alpha(\varphi))''|(x) \le 
\frac{B_1}{2x}
X_\alpha(x)( \NN_\alpha(\varphi))'(x) \, , 
$$
(which follows from $\NN_\alpha(\varphi)'' (x)\le b_2\NN_\alpha(\varphi)(x)/x^2\le -
b_2 \NN_\alpha(\varphi)'(x)/(\bar b_1 x)$
and  \eqref{Xb}, \eqref{X'b}).

Next, if
$0<\alpha<\beta <1/2$,
by using 
$\CC_{1,*}(\alpha)\subset \CC_{1,*}(\beta)$ and  \eqref{toremember1},  we get summability of the first term of the
expression in the right-hand side of \eqref{calc0} as $k\to \infty$: 
$$
\left|\sum_{j=0}^{k-1} \int_0^1 \psi \LL_\beta^{j} ([X_\alpha \NN_{\alpha}( \LL_\alpha^{k-j-1}({\bf 1}))] ')\, dx
\right| \le 
C_\beta \|\psi\|_{L^\infty} \sum_{j=0}^{k-1} \frac{(\log j)^{1/\beta}}{j^{-1+1/\beta}}  \, .
$$

When we consider the second
term of the right-hand side 
of \eqref{calc0}, we require the derivatives
of order three in $\CC_3$.  Applying \eqref{magic}
to $\varphi=\LL_\alpha^{k-j-1}(\mathbf{1})\in \CC_3\cap \CC_{*,1}(\alpha)$, and using Proposition ~\ref{cheat'},
we see  that for any $\alpha \le \gamma\le \beta$, up to taking larger $a$ and $b_1$ (uniformly
 in $1\le j \le k-1$) the decomposition \eqref{seven} of
$
\partial^2_{\gamma} \LL_{\gamma}( \LL_\alpha^{k-j-1}({\bf 1}))$
gives seven functions which, up to multiplying by $-1$ and
adding a uniformly bounded constant, all lie  in 
$\CC_{*,1}(\alpha, a, b_1)
\subset \CC_{*,1}(\beta)$. We proceed
as in the proof of \eqref{coneclaim}, developing
the Leibniz inequality. We shall focus on
the contribution of $X_\alpha^2 (\NN_\alpha(\varphi))''$, leaving the
other terms to the reader. We need to check that
$$
|2 X_\alpha X_ \alpha' (\NN_\alpha(\varphi))''
+ X_\alpha^2 (\NN_\alpha(\varphi))'''|
\le \frac{B_1}{7x} |X_\alpha^2 (\NN_\alpha(\varphi))''|\, .
$$
The above bound follows from 
$(\NN_\alpha(\varphi))'''(x)\le b_3 \NN_\alpha(\varphi)(x)/x^3
\le  b_3 (\NN_\alpha(\varphi))''(x)/(\bar b_2 x)$
and \eqref{Xb}, \eqref{X'b}.
 Since we are in a cone, we may apply \eqref{toremember1}
once more, we thus get the  bound
\begin{equation}\label{better'}
\frac{\|\psi\|_{L^\infty}}{\beta- \alpha}
\int_{\alpha}^{\beta} (\gamma-\alpha) C_\beta
\sum_{j=0}^{k-1} \frac{(\log j)^{1/\beta}}{j^{-1+1/\beta}} \, d\gamma\le C_\beta \|\psi\|_{L^\infty} (\beta -\alpha)\, .
\end{equation}

Step~2 does not change, and the proof of Theorem~\ref{main} bypassing \cite{Go, Goth}
is complete.


\begin{thebibliography}{KKPW}
 

\bibitem[B1]{B1} V. Baladi,{\it On the susceptibility function of piecewise expanding interval
              maps,} Comm. Math. Phys. \textbf{275}  (2007) 839--859.

\bibitem[B2]{B2} V. Baladi, {\it  Linear response despite critical points,} Nonlinearity 
\textbf{21} (2008) T81--T90.

 \bibitem[B3]{B3} V. Baladi,  {\it Linear response, or else,} ICM Seoul 2014,  Proceedings, Volume III, 525--545, {\tt http://www.icm2014.org/en/vod/proceedings.}
 
 \bibitem[BMS]{BMS} V. Baladi, S. Marmi, and D. Sauzin {\it Natural boundary for the susceptibility function of generic
piecewise expanding unimodal maps,}
Ergodic Theory Dynam. Systems \textbf{10}
(2013) 1--24.

\bibitem[BS]{BS} V. Baladi and D. Smania, 
\emph{Linear response formula for piecewise expanding unimodal maps,} Nonlinearity \textbf{21}
  (2008) 677--711. (Corrigendum:   Nonlinearity \textbf{25} (2012) 2203--2205).
 
\bibitem[BCV]{BomCasVar12}  T. Bomfim, A. Castro, and P. Varandas,
\emph{Differentiability of thermodynamical quantities in non\--uni\-formly expanding dynamics,}
arXiv:1205.5361, to appear Adv. Math.
 
 \bibitem[BT]{BT} H. Bruin and M. Todd, 
 {\it Equilibrium states for potentials with $\sup \phi -\inf \phi < h_{top}(f)$,} Comm. Math. Phys. 
 {\bf 283} (2008) 579--611.
 
 \bibitem[CD]{CD} F. Contreras and D. Dolgopyat, {\it Regularity of absolutely continuous invariant measures for piecewise expanding unimodal maps,}   arXiv:1504.04214.

 \bibitem[Do]{Do} D. Dolgopyat, 
{\it On differentiability of {SRB} states for partially hyperbolic systems,
}
Invent. Math. {\bf 155}  (2004) 389--449.

\bibitem[FT]{FT}
J.M. Freitas and M. Todd,
{\it   Statistical stability of equilibrium states for interval maps,} Nonlinearity 
{\bf 22} (2009) 259--281.

\bibitem[Go]
{Go} S. Gou\"ezel, {\it Sharp polynomial estimates for the decay of correlations,}
Israel J. Math. {\bf 139} (2004) 29--65.

\bibitem[Goth]
{Goth} S. Gou\"ezel, {\it Vitesse de d\'ecorr\'elation et th\'eor\`emes limites pour les applications
non uniform\'ement dilatantes,}
PhD thesis, Orsay, 2004.
 
 \bibitem[HM]{HM} M. Hairer and A.J. Majda,
 {\it A simple framework to justify linear response theory,} Nonlinearity
 {\bf 23}  (2010) 909--922.
 
\bibitem[KKPW]{KKPW} A. Katok,  G. Knieper, M. Pollicott,  and H. Weiss, 
{\it Differentiability and analyticity of topological entropy for {A}nosov and geodesic flows,}
Invent. Math. \textbf{98} (1989) 581--597.


\bibitem[Ko]{Ko} A. Korepanov, {\it Linear response for intermittent maps with summable
and non-summable decay of correlations,}
arXiv:1508.06571.

\bibitem[LSV]{LSV}
C. Liverani, B. Saussol, and S. Vaienti,
{\it  A probabilistic approach to intermittency,}
Ergodic Theory Dynam. Systems {\bf 19} (1999) 671--685.

\bibitem[Lu1]{Luca} V. Lucarini, D. Faranda, J. Wouters,  and T. Kuna, 
{\it Towards a general theory of extremes for observables of chaotic dynamical systems,}
 J. Stat. Phys. {\bf 154} (2014) 723--750. 


\bibitem[Lu2]{Luca2} V. Lucarini et al., 
{\it Extremes and Recurrence in Dynamical Systems,}
 John Wiley and Sons, 2015.

\bibitem[Ma]{Ma} M. Mazzolena, 
\emph{Dinamiche espansive unidimensionali:
dipendenza della misura invariante da un parametro,} Master's Thesis, Roma 2
     (2007).
 

 
\bibitem[Sa]{Sa} O. Sarig, {\it Subexponential decay of correlations,}
 Invent. Math. \textbf{150} (2002) 629--653. 

\bibitem[Ru]{Ru} D. Ruelle, 
{\it Differentiation of SRB states,}
Comm. Math. Phys. \textbf{187} (1997) 227--241.


\bibitem[Ru1]{Ru1} D. Ruelle,  {\it General linear response formula in statistical mechanics, and the fluctuation-dissipation theorem far from equilibrium,} Phys. Lett. A \textbf{245} (1998) 220--224.

\bibitem[Ru2]{Ru2} D. Ruelle, \emph{Structure and $f$-dependence of the {A.C.I.M.} for a unimodal map $f$  of {M}isiurewicz type,} Comm. Math. Phys.,
    \textbf{287}
   (2009) 1039--1070.
   
\bibitem[Ru3]{Ru3}    D. Ruelle,  {\it A review of linear response theory for general differentiable dynamical systems,} Nonlinearity \textbf{22} (2009) 855--870.


\bibitem[Th1]{Th} M. Thaler,
{\it Estimates of the invariant densities of endomorphisms with indifferent fixed points,}
 Israel J. Math. \textbf{37} (1980) 303--314.

\bibitem[Th2]{Tha00} M. Thaler,
\emph{The asymptotics of the {P}erron-{F}robenius operator of a class of interval maps preserving infinite measures,}
 Studia Math. \textbf{143} (2000) 103--119.

\end{thebibliography}
\end{document}